\documentclass[twosided, 12pt]{article}
\usepackage{geometry}
\usepackage{import}
\usepackage{glossaries}
\usepackage{float}
\usepackage{titlesec}
\usepackage{amsmath, amsfonts, bbm, mathrsfs, graphicx, amssymb, amsthm, hyperref, centernot, enumerate, bbm, xcolor, lmodern, mathdots, wasysym, mathtools,tikz}
\usepackage{tcolorbox}
\usepackage{cleveref}
\usepackage{verbatim}
\usepackage[inline]{enumitem}
\usepackage{parskip}
\usepackage [english]{babel}
\usepackage [autostyle, english = american]{csquotes}
\usepackage[T1]{fontenc}  
\usepackage{enumerate}
\setitemize[1]{label=\textbullet}
\setenumerate[1]{label=(\roman*), itemsep=0pt, topsep=-10pt, leftmargin=0.65cm}
\usepackage{bm}

\usepackage{systeme}
\usepackage{amssymb,latexsym,graphicx}
\usepackage{amsmath,amscd,amsthm}
\usepackage{mathrsfs}
\usepackage{graphicx}
\usepackage{esint}
\usepackage{pifont}
\usepackage{subcaption}
\usepackage{breqn}
\usepackage{environ}
\usepackage{pgfplots}
\usepackage{math tools}
\numberwithin{equation}{section}
\theoremstyle{plain}
\usepackage{relsize}
\usepackage{hyperref}
\hypersetup{
    colorlinks=true,
    linkcolor=black,
    filecolor=black,      
    urlcolor=black,
    citecolor=black,
}
\newtheorem{thm}{Theorem}[section]
\newtheorem*{thm*}{Theorem}

\newtheorem{prop}[thm]{Proposition}

\newtheorem{cor}[thm]{Corollary}

\newtheorem{defn}[thm]{Definition}
\newtheorem{lemma}[thm]{Lemma}

\newcommand{\diff}{\mathrm{d}}

\newcommand{\eps}{\epsilon}

\providecommand{\keywords}[1]
{
  \small	
  \textbf{\textit{Keywords--}} #1
}

\providecommand{\msc}[1]
{
  \small	
  \textbf{\textit{MSC classes--}} #1
}

\title{Existence-Uniqueness Theory and Small-Data Decay for a Reaction-Diffusion Model of Wildfire Spread}
\author{A. George Morgan\footnote{Department of Mathematics,
University of Toronto,
40 St. George St., Room 6290,
Toronto, Ontario, CA,
M5S 2E4. Institutional email: 
adam.morgan@mail.utoronto.ca}
}

\date{}

\begin{document}

\maketitle

 \begin{abstract}


 \noindent 
I examine some analytical properties of a nonlinear reaction-diffusion system that has been used to model the propagation of a wildfire. I establish global-in-time existence and uniqueness of bounded mild solutions to the Cauchy problem for this system given bounded initial data. In particular, this shows that the model does not allow for thermal blow-up. If the initial temperature and fuel density also satisfy certain integrability conditions, the $L^2$-norms of these global solutions are uniformly bounded in time. Additionally, I use a bootstrap argument to show that small initial temperatures give rise to solutions that decay to zero as time goes to infinity, proving the existence of initial states that do not develop into travelling combustion waves. 
  \end{abstract}

\noindent

\noindent
\msc{35K57, 80A25}


\section{Introduction}
\par In this article, I consider a reaction-diffusion model of wildfire propagation in a fuel layer with a very large spatial extent. Fix a spatial dimension $d\geq 1$. Consider a certain amount of solid fuel (trees, shrubberies, and so on) spread throughout $\mathbb{R}^d$ with dimensionless mass density at spatial location $x$ and time $t$ given by $Y(x,t)$. A fire with a dimensionless temperature field $T(x,t)$ is moving through space and consuming this fuel: temperature is normalized so that $T=0$ is the ambient temperature of the system's environment. Also, we suppose the reaction rate governing combustion is given by the Arrhenius law 
\begin{equation}\label{eqn:reaction_rate}
r(T) = \begin{cases}
e^{-\frac{1}{T}} & \quad T>0
\\
0 & \quad T\leq 0.
\end{cases}
\end{equation}
Under these assumptions, Weber et al. \cite{WMSG1997} were the first to model the evolution of $T$ and $Y$ using the following Cauchy problem: given two parameters $\lambda, \beta\geq0$, 
\begin{equation}\label{eqn:model_ivp_full}
\left\{
\begin{aligned}
    T_{t} &=  \left(\Delta - \lambda\right) T + Y r\left(T\right)
    \\
    Y_{t} &= -\beta Y r\left(T\right)
    \\
    T|_{t=0} &= T_0(x)
    \\
    Y|_{t=0} &= Y_0(x). 
\end{aligned}
\right.
\end{equation}
We always assume $Y_0(x)\geq 0$ everywhere. Up to a different choice of reaction rate, these are precisely the equations of first-order solid-fuel combustion at infinite Lewis number, with an added heat loss term \cite[\S 1.3]{BE1989}. Strictly speaking, in \cite{WMSG1997} and other work $Y(x,t)$ is actually the \emph{mass fraction} of fuel. However, I want to allow for an infinite initial amount of fuel, and in this situation the mass fraction does not make sense and it is more natural to regard $Y(x,t)$ as a mass density. For an instructive derivation of \eqref{eqn:model_ivp_full} and an explanation of why it is useful for data assimilation, see Mandel et al. \cite{Mandel_etal_2008}. The discussion in Mandel et al. highlights that $r(0)=0$ implies no combustion takes place at the ambient temperature $T=0$. Consequently, \eqref{eqn:model_ivp_full} admits solutions that take the form of travelling combustion waves; this had already been confirmed using formal asymptotics and numerics in \cite{WMSG1997}. 
\par In light of the existence of travelling wave solutions, \eqref{eqn:model_ivp_full} has attracted attention from researchers in both mathematical modelling and dynamical systems theory. On the more concrete side, \eqref{eqn:model_ivp_full} has been used to develop computationally efficient reduced-order models of wildfire dynamics \cite{BSU2021, LBGLL2016}. Also, Johnston \cite{Johnston2022} applied \eqref{eqn:model_ivp_full} in a modelling pipeline where remote sensing data was leveraged to estimate realistic model parameter values. Additionally, variants of the model incorporating the effects of wind \cite{BBH2009, Mandel_etal_2008} and radiative heat transfer \cite{AF2002, MPR2024, RNO2024, Weber1991} have also been considered. When it comes to rigorous results on the behaviour of solutions to \eqref{eqn:model_ivp_full}, Billingham \cite{Billingham2000} proved the existence of travelling waves given suitable initial data and $d=1$. Linear stability of these travelling waves was examined numerically by Varas and Vega \cite{VV2002}. Existence and linear stability of travelling waves for related combustion models in $d=1$ with \emph{finite} Lewis number has also been considered, see for example \cite{GJ2009, Razani2004}. There is certainly no shortage of other interesting work on travelling waves in wildland fire propagation and other solid-fuel systems, but a complete literature review is beyond the scope of this brief article. 
\par While there is plenty of mathematical work on the properties of travelling wave solutions to \eqref{eqn:model_ivp_full}, to my knowledge there are no published results on the existence and uniqueness of global-in-time solutions to \eqref{eqn:model_ivp_full} for \emph{generic} bounded initial data (see, however, the existence result for a related model in \cite[\S 4.4]{BE1989}). In particular, with the exception of the lumped-parameter analysis in \cite{RNO2024}, I have not found any estimates that preclude finite-time blow-up of $|T(x,t)|$. Thermal blow-up phenomena indeed arise in certain solid-fuel combustion models \cite[\S 3.2]{BE1989}
, so understanding whether or not blow-up occurs in solutions to \eqref{eqn:model_ivp_full} is a meaningful question. Also, in \cite{Mandel_etal_2008}, the authors use a heuristic lumped-parameter argument to predict that a fire with a sufficiently low temperature eventually burns itself out: in the jargon of combustion theory, this means the model predicts a nonzero ``auto-ignition temperature''. I have not, however, found any rigorous proof of the claim that \eqref{eqn:model_ivp_full} has a nonzero auto-ignition temperature. 
\par In this article, I prove three novel results to address the knowledge gaps discussed in the previous paragraph:
\par \noindent 
\begin{enumerate}
    \item for any bounded initial state $\left(T_0, Y_0\right)$ there exists a unique, bounded mild solution to \eqref{eqn:model_ivp_full} (to be defined precisely below) valid on the entire time interval $[0,\infty)$, provided $\lambda >0$: in particular, thermal blow-up is impossible (corollary \ref{cor:linfty_gwp});
    \item further, if $Y_0$ is also integrable and $T_0$ is also square-integrable, then the $L^2_x$ norm of $T(x,t)$ is bounded uniformly in time (corollary \ref{cor:linfty_gwp} as well);
    \item given an initial state $T_0$ that is small in a particular physically reasonable norm, the solution to \eqref{eqn:model_ivp_full} satisfies $\lim_{t\rightarrow 0}|T(x,t)| = 0$ uniformly in $x$ for every $\lambda \geq 0$: morally, this confirms that \eqref{eqn:model_ivp_full} indeed has a nonzero auto-ignition temperature, and that not \emph{every} initial state gives rise to a travelling wave solution (theorem \ref{thm:decay_for_lambda=0_nonradiating});
\end{enumerate} 
\ \par \noindent 
I hope these results prove useful to further investigations into the \emph{nonlinear} stability of travelling waves and to the development of rigorous error estimates for numerical discretizations of \eqref{eqn:model_ivp_full}. 
\subsection{Notation}
\noindent For definitions of any of the objects appearing below, see \cite{Evans2010}.  
\begin{itemize}
\item $d$ always denotes a positive integer greater than or equal to $1$. 
\item $\mathbb{N}_0=\left\{0,1,2,3,...\right\}$.
\item $C$ always denotes some nonnegative constant that may change from line to line. If $\alpha$ is some real parameter, then $C(\alpha)$ denotes a constant depending on $\alpha$. If $a,b\in \mathbb{R}$, we say that $a\lesssim b$ if there exists $C$ such that $a\leq bC$. If $a,b\in \mathbb{R}$, we say that $a\simeq b$ if $a\lesssim b$ and $b\lesssim a$. 
\item For $k\in \mathbb{N}_0$, and a nice function $\varphi\colon \mathbb{R}^d\rightarrow \mathbb{R}$, $D^{k}_{x}\varphi$ denotes the set of all order $k$ partial derivatives of $\varphi$. 
    \item For $p\in \left[1,\infty\right]$, $L^{p}_{x}\left(\mathbb{R}^{d}\right)$ denotes the usual Lebesgue space of real-valued functions of $x\in \mathbb{R}^d$. 
   \item For $k\in \mathbb{N}_0$, $H^{k}_{x}\left(\mathbb{R}^d\right)$ denotes the $L^2_{x}$-based inhomogeneous order $k$ Sobolev space of real-valued functions on $\mathbb{R}^d$, and $\dot{H}^{k}_{x}\left(\mathbb{R}^d\right)$ denotes the $L^2_{x}$-based homogeneous order $k$ Sobolev space of real-valued functions on $\mathbb{R}^d$. 
   \item Given a Banach space $A$ and a time interval $[0,t_{*}]$, we let $C^0_t A$ denote the space of continuous curves $[0,t_{*}] \rightarrow A$. When the value of $t_{*}$ has a substantial impact on the topic at hand, this value will be made clear. 
    \end{itemize}
\section{Results}
\subsection{Preliminaries}
\noindent First, we recall some decay and smoothing estimates for the linear heat equation.
\begin{lemma}\label{lemma:heat_flow_est}
\par \noindent
   \begin{enumerate}
       \item Fix any $p\in [1,\infty]$. For any $\varphi\in L^p_{x}$,
    \begin{equation}\label{eqn:weak_damping_estimate}
        \left\|e^{t\Delta}\varphi\right\|_{L^p_{x}} \leq \left\|\varphi\right\|_{L^p_{x}} \quad \forall \ t>0. 
    \end{equation}
    \item For any $p,q,\ell\in \left[1,\infty\right]$ with $p^{-1}=q^{-1}-\ell^{-1}$, any $\varphi\in L^q_{x}$, and any $k\in \mathbb{N}_0$, we have 
    \begin{equation}\label{eqn:smoothing_estimate}
        \left\|D^{k}_{x} e^{t\Delta}\varphi\right\|_{L^p_x} \lesssim t^{-\frac12\left(\frac{d}{\ell}+k\right)}\left\|\varphi\right\|_{L^q_{x}}  \quad \forall \ t>0. 
    \end{equation}
   \end{enumerate} 
\end{lemma}
\begin{proof}
All these bounds follow from the time-decay properties of the heat kernel on $\mathbb{R}^d$; for example, see \cite[pp. 41-42]{LP2015} for an outline of how to prove \eqref{eqn:smoothing_estimate}.  
\end{proof}

 \noindent Additionally, we need a result about the function $r(T)$ defined in \eqref{eqn:reaction_rate}.
\begin{lemma}\label{lemma:reaction_rate}
    For any $p\in \mathbb{N}_0$, there exists a constant $C(p)$ such that 
    $$
    r(T) \leq C(p) T^{p} \quad \forall \ T\in [0,1].
    $$
\end{lemma}
\begin{proof}
    This is obvious since $r(T)$ vanishes to all orders at $T=0$: it is the canonical example of a smooth, non-analytic function. 
\end{proof}

\subsection{Mild Solutions and their Basic Properties}
\noindent We may use Duhamel's principle to re-formulate \eqref{eqn:model_ivp_full} in a way that lets us introduce solutions with limited regularity: 
\begin{defn} Let $A, A_0$, and $B$ be Banach subspaces of $L^{\infty}_{x}$, with $A \subseteq A_0$. Suppose we are given an initial datum $\left(T_0(x), Y_0(x)\right)\in A_0 \times B$. We say that $\left(T(x,t), Y(x,t)\right)\in C^0_t A \times C^0_t B$ is a \textbf{mild solution} of \eqref{eqn:model_ivp_full} if this pair satisfies
\begin{subequations}\label{eqn:mild_soln}
\begin{align}
    T(x,t) &= e^{-\lambda t}e^{t\Delta}T_0 + \int_{0}^{t} e^{-\lambda\left(t-\tau\right)}e^{\left(t-\tau\right)\Delta} Y(x,\tau)r\left(T(x,\tau)\right) \ \diff \tau
    \\
    Y(x,t) &= Y_0(x) \exp\left(-\beta\int_0^t r\left(T(x,\sigma)\right) \ \diff \sigma\right). 
\end{align}
\end{subequations}
\end{defn}
\noindent Below, the time interval on which a mild solution is valid will be specified when it's vital to know. Immediately, one sees that mild solutions have some important properties. 
\par \noindent 
\begin{enumerate}
    \item First, if $Y_0(x)\geq 0$ everywhere then $Y(x,t)\geq 0$ also holds. This makes perfect sense: an initially nonnegative fuel density should never become negative. 
    \item Second, owing to the smoothing properties of the heat flow $e^{t\Delta}$, $T(x,t)$ is smooth in $x$ for every $t>0$. However, $Y(x,t)$ is only as smooth as $Y_0(x)$.
\end{enumerate}

\subsection{\textit{A Priori} Estimates and Existence-Uniqueness Theory}
\noindent In this subsection I prove some basic \textit{a priori} bounds on the Lebesgue and Sobolev norms of mild solutions to \eqref{eqn:model_ivp_full} with $\lambda>0$. These bounds are then used to help construct global-in-time mild solutions of \eqref{eqn:model_ivp_full}. We start with the easiest bound possible: 
\begin{prop} \label{prop:fuel_decay} For any $p\in [1,\infty]$, if $(T,Y)$ is a mild solution of \eqref{eqn:model_ivp_full}, then 
\begin{equation}
    \sup_{t>0}\left\|Y(x,t)\right\|_{L^p_{x}} \leq  \left\|Y_0(x)\right\|_{L^p_{x}}
\end{equation}
provided the right-hand side is finite.
\end{prop}
\begin{proof}
Since $r(T)$ is nonnegative, $Y(x,t)\leq Y_0(x)$ almost everywhere. 
\end{proof}
\noindent According to Mandel et al. \cite[\S 4.1]{Mandel_etal_2008}, $Y(x,t)$ is expected to settle to a nonzero steady value for $t\gg 1$. I therefore do not think one can prove a variant of the above result that includes a time-dependent decay factor on the right-hand side. 
\begin{prop}\label{prop:linfty_est}
Suppose $\lambda>0$. If we are given an initial datum $(T_0, Y_0)\in L^{\infty}_{x}\times L^{\infty}_{x}$ and a corresponding mild solution $(T,Y)\in C^{0}_{t}L^{\infty}_{x}\times C^{0}_{t}L^{\infty}_{x}$ of \eqref{eqn:model_ivp_full}, then 
\begin{equation}
 \left\|T(x,t)\right\|_{L^{\infty}_{x}} \leq e^{-\lambda t}\left\|T_0(x)\right\|_{L^{\infty}_{x}} + \lambda^{-1}\left\|Y_0(x)\right\|_{L^{\infty}_{x}}. 
\end{equation}
\end{prop}
\begin{proof}
    Using lemma \ref{lemma:heat_flow_est} and $\left|r(T)\right|\leq 1$, we have
    \begin{align*}
        \left\|T\right\|_{L^{\infty}_{x}} &\leq e^{-\lambda t} \left\|T_0\right\|_{L^{\infty}_{x}} + \sup_{t>0}\left\|Y r\left(T\right)\right\|_{L^\infty_x}\int_0^t e^{-\lambda \left(t-\tau\right)} \ \diff \tau
        \leq e^{-\lambda t} \left\|T_0\right\|_{L^{\infty}_{x}} + \lambda^{-1}\left(1-e^{-\lambda t}\right)\left\|Y_0\right\|_{L^\infty_x}. 
    \end{align*}
\end{proof}

\begin{prop} \label{prop:l2_est} 
    Suppose $\lambda >0$. If we are given an initial datum $(T_0, Y_0)\in \left(L^{2}_{x}\cap L^{\infty}_{x}\right)\times L^{\infty}_{x}$ and a corresponding mild solution $(T,Y)\in C^{0}_{t}\left(L^{2}_{x}\cap L^{\infty}_{x}\right)\times C^{0}_{t}L^{\infty}_{x}$ of \eqref{eqn:model_ivp_full}, then
    \begin{equation}\label{eqn:general_L2_estimate}
        \left\|T\right\|_{L^2_{x}}^2 \lesssim \min\left\{\left\|T_0\right\|_{L^2_{x}}^2 + C\left(\left\|T_0\right\|_{L^{\infty}_{x}}, \left\|Y_0\right\|_{ L^{\infty}_{x}}, \lambda, \beta\right)\left\|Y_0- Y\right\|_{L^1_{x}}, \ \left\|T_0\right\|_{L^2_{x}}^2 e^{\left(\left\|Y_0\right\|_{L^{\infty}_{x}}-\lambda\right)t}\right\}
    \end{equation}
    In particular, if $(T_0, Y_0)\in \left(L^{2}_{x}\cap L^{\infty}_{x}\right)\times \left(L^{1}_{x}\cap L^{\infty}_{x}\right)$ then
    \begin{equation}\label{eqn:special_L2_estimate}
        \sup_{t>0}\left\|T(x,t)\right\|_{L^2_{x}} \leq C\left(\left\|T_0\right\|_{L^2_{x}\cap L^{\infty}_{x}}, \left\|Y_0\right\|_{L^1_{x}\cap L^{\infty}_{x}}, \lambda, \beta\right). 
    \end{equation}
\end{prop}
\begin{proof}
    Recall that $T(x,t)$ is smooth in $x$, so we can write
    $T_{t} = \left(\Delta-\lambda\right) T + Yr\left(T\right)$
    where equality is understood in an $L^2_x$ sense. We then use a standard energy argument:  
    \begin{align}
    \frac{\diff}{\diff t}\frac12\left\|T\right\|_{L^2_{x}}^2 &= - \left\|T\right\|^2_{\dot{H}^1_{x}} -\frac{\lambda}{2}\left\|T\right\|^2_{L^2_x}+ \int T(x,t) \ Y(x,t) \ r\left(T(x,t)\right) \ \diff x \nonumber
    \\
    &\leq  \int T(x,t) \ Y(x,t) \ r\left(T(x,t)\right) \ \diff x \label{eqn:funny_line}
    \\
    &= \beta^{-1} \int T(x,t) \ \left(-Y_{t}(x,t)\right) \ \diff x \nonumber
    \\
    &\leq \beta^{-1}\sup_{t>0}\left\|T(x,t)\right\|_{L^{\infty}_{x}}\int  \left(-Y_{t}(x,t)\right) \ \diff x. \nonumber
    \end{align}
    Integrating both sides of the above over $[0,t]$ and using proposition \ref{prop:linfty_est} gives 
    \begin{align*}
        \left\|T(x,t)\right\|_{L^2_{x}}^2 \lesssim \left\|T_0(x)\right\|_{L^2_{x}}^2 + C\left(\left\|T_0\right\|_{L^{\infty}_{x}}, \left\|Y_0\right\|_{ L^{\infty}_{x}}, \lambda, \beta\right)\int_0^t\int \left(-Y_{\tau}(x,\tau)\right) \ \diff x \ \diff \tau.
    \end{align*}
    Then, we use Fubini's theorem and the fundamental theorem of calculus to obtain one part of \eqref{eqn:general_L2_estimate}. The special case \eqref{eqn:special_L2_estimate} follows since $Y_0(x) - Y(x,t) \leq Y_0(x)$. For the other part of \eqref{eqn:general_L2_estimate}, we keep the $\lambda$ term around, instead use $Tr(T)\lesssim T^2$ in \eqref{eqn:funny_line}, place $Y$ in $L^{\infty}_{x}$, and conclude with an application of Gr\"{o}nwall's inequality. 
\end{proof}
\par Next, I demonstrate how these \textit{a priori} estimates may be applied to obtain global-in-time mild solutions to \eqref{eqn:model_ivp_full}. 
\begin{cor}\label{cor:linfty_gwp}
      Suppose $\lambda>0$ and we are given an initial datum $\left(T_0, Y_0\right)\in L^{\infty}_{x}\times L^{\infty}_{x}$. Then, \eqref{eqn:model_ivp_full} admits a unique mild solution $(T,Y)\in C^0_t L^{\infty}_{x}\times C^0_t L^{\infty}_{x}$ valid for $t\in [0,\infty)$. Further, if we also have $T_0\in L^2_{x}$ and $Y_0\in L^1_x$, then the mild solution obeys \eqref{eqn:special_L2_estimate}. 
\end{cor}
\begin{proof}
    Fixed-point iteration (see for example \cite{Tao2006}) can be used to established local-in-time existence and uniqueness. Proposition \ref{prop:linfty_est} allows us to extend this local solution to a global one. With this solution in hand, we apply proposition \ref{prop:l2_est} to conclude.
\end{proof}
\noindent Consequently, solutions to \eqref{eqn:model_ivp_full} do not exhibit thermal blow-up when $\lambda>0$. In the case $\lambda=0$ the techniques from proposition \ref{prop:linfty_est} would yield
\begin{equation*}
 \left\|T(x,t)\right\|_{L^{\infty}_{x}} \leq\left\|T_0(x)\right\|_{L^{\infty}_{x}} + t\left\|Y_0(x)\right\|_{L^{\infty}_{x}}. 
\end{equation*}
While this does not give a nice uniform-in-time bound, it at least says that thermal blow-up does not occur in finite time irrespective of $\lambda$. 
\par Even though the mild solutions constructed above are smooth with respect to $x$, they are not necessarily classical solutions (that is, they are not also continuously differentiable in $t$). Indeed, the PDE $T_t = \Delta T + Yr(T)$ tells us that $T_t$ only has the spatial regularity of $Y$, itself controlled by the regularity of $Y_0$. However, if $Y_0$ is continuous, then so is $T_t$, and we indeed have a classical solution. In particular, the choice $Y_0\equiv 1$ used in \cite{Mandel_etal_2008} gives rise to a unique global classical solution.
\par For the sake of completeness, I mention that we can obtain coarse estimates on $\nabla T$ without too much trouble:
\begin{prop}
Pick $p,q,\ell \in \left[1,\infty\right]$ with $p^{-1}=q^{-1}-\ell^{-1}$ and $\ell>d$. If $(T_0, Y_0)\in \left(L^{p}_{x}\cap L^{\infty}_{x}\right)\times \left(L^{q}_{x}\cap L^{\infty}_{x}\right)$ then the corresponding mild solution  $(T,Y)$ of \eqref{eqn:model_ivp_full} satisfies
    \begin{equation}\label{eqn:special_smoothing_est}
     \left\|\nabla T(x,t)\right\|_{L^p_{x}} \lesssim t^{-\frac12} \left\|T_0(x)\right\|_{L^p_x} + t^{\frac12\left(1-\frac{d}{\ell}\right)}\left\|Y_0\right\|_{L^q_x} \quad \forall \ t>0.
    \end{equation}
In particular, 
 \begin{equation}\label{eqn:linfty_smoothing_est}
     \left\|\nabla T(x,t)\right\|_{L^\infty_{x}} \lesssim t^{-\frac12} \left\|T_0(x)\right\|_{L^\infty_x} + t^{\frac12}\left\|Y_0\right\|_{L^\infty_x} \quad \forall \ t>0.
    \end{equation}
\end{prop}
\begin{proof}
 Using \eqref{eqn:smoothing_estimate} and $|r(T)|\leq 1$, we have 
    \begin{align*}
        \left\|\nabla T(x,t)\right\|_{L^{p}_{x}} 
        &\lesssim t^{-\frac{1}{2}}\left\|T_0\right\|_{L^p_{x}} + \left\|Y_0\right\|_{L^q_x}\int_0^t \left(t-\tau\right)^{-\frac{1}{2}\left(\frac{d}{\ell}+1\right)} \ \diff \tau
    \end{align*}
    and we're all done upon computing the integral. 
\end{proof}
\noindent Thus solutions of \eqref{eqn:model_ivp_full} do not develop sharp edges in finite time. I expect that one can get rid of the temporally growing term in \eqref{eqn:special_smoothing_est} using a more thorough analysis.
\subsection{Decay of Small-Data Solutions}
\noindent Next I show that, if the initial temperature $T_0$ is sufficiently small, then $\lim_{t\rightarrow \infty}|T(x,t)|=0$ uniformly. This confirms that \eqref{eqn:model_ivp_full} has a nonzero auto-ignition temperature up to some mild assumptions on $T_0$. From our work in the last section, we know that $\lambda=0$ gives the worst $L^{\infty}_x$ temperature bounds possible, so in this section we'll set $\lambda=0$ for concreteness: if we can establish decay with $\lambda=0$ then it is even simpler to establish decay with $\lambda>0$. While this means we don't have access to the existence-uniqueness results from the previous section, we can still build a temporally-decaying global solution using the assumption of small initial data and the bootstrap principle. This strategy for establishing time decay is well-known in the theory of dispersive PDEs, see for instance \cite{Morgan2024}. 
\par Throughout this section, we suppose that $\eps_0\in (0,1)$ is a small parameter (to be constrained later) and 
\begin{equation}\label{eqn:smallness_hypothesis_4_burnout}
    \left\|T_0\right\|_{L^1_x\cap L^{\infty}_x} < \eps_0 \quad \text{and} \quad  \left\|Y_0\right\|_{L^{\infty}_x} < \infty. 
\end{equation}
The requirement that $T_0$ is small in $L^1_x$ is a minor technical one giving us access to a decay estimate.  Now, given $t_{*}\in (0,\infty]$ and a sufficiently nice function $T(x,t)$ defined for $t\in [0,t_{*}]$, we set 
\begin{equation*}
    \left\|T(x,t)\right\|_{X_{t_{*}}} \doteq \sup_{t\in [0,t_{*}]}\left[ \left(1+t\right)^{\frac{d}{2}}\left\|T(x,t)\right\|_{L^{\infty}_{x}}+ \left\|T(x,t)\right\|_{L^1_x}\right].
\end{equation*}
This norm naturally gives rise to a Banach space $X_{t_{*}}$ defined by
\begin{equation*}
    X_{t_{*}} \doteq \left\{T\colon [0,t_{*}] \rightarrow L^1_x \cap L^\infty_{x}\  \text{continuous}\ \bigg| \  \left\|T(x,t)\right\|_{X_{t_{*}}}  < \infty\right\}. 
\end{equation*}
\par The first order of business is to prove a new \textit{a priori} estimate for solutions that are small in $X_{t_*}$. 
\begin{prop}\label{prop:bootstrap_close}
    Take $\eps_1 \in \left[\eps_0, 1\right)$. Assume that \eqref{eqn:smallness_hypothesis_4_burnout} holds and that there is a corresponding mild solution $\left(T,Y\right)\in X_{t_{*}}\times C^0_t L^{\infty}_{x}$ to \eqref{eqn:model_ivp_full} valid on $[0,t_{*}]$. Further, assume that $\left\|T\right\|_{X_{t_{*}}} \leq \eps_1$. Then, there exists an absolute constant $C>0$ such that 
    \begin{equation}
        \left\|T(x,t)\right\|_{X_{t_{*}}} \leq 2\eps_0 + C\left\|Y_0\right\|_{L^\infty_x} \eps_1^{11}. 
    \end{equation}
\end{prop}
\begin{proof}
  Let's look at the $L^{\infty}_{x}$ part of the $X_{t_{*}}$ norm first. Together with the hypothesis of our claim, lemma \ref{lemma:reaction_rate} (with $p=11$) implies $\left\|r(T)\right\|_{L^\infty_x}\leq C\eps_1^{11}\left(1+t\right)^{-\frac{11}{2}d}$ and $\left\|r(T)\right\|_{L^1_x}\leq C\eps_1^{11}\left(1+t\right)^{-5d}$. Combining these estimates with \eqref{eqn:weak_damping_estimate}, \eqref{eqn:smoothing_estimate}, and \eqref{eqn:smallness_hypothesis_4_burnout}, we have for $t> 1$
    \begin{align*}
        \left\|T(x,t)\right\|_{L^\infty_x}
        &\leq \left(1+t\right)^{-\frac{d}{2}} \eps_0 + \int_0^t \left\|e^{\left(t-\tau\right)\Delta} Y(x,\tau)r\left(T(x,\tau)\right)\right\|_{L^\infty_x} \ \diff \tau
        \\
        &\leq \left(1+t\right)^{-\frac{d}{2}} \eps_0 + \int_0^{t/2} \left(t-\tau\right)^{-\frac{d}{2}} \left\|Y(x,\tau) \ r\left(T(x,\tau)\right)\right\|_{L^1_{x}} \ \diff \tau
        + \int_{t/2}^t \left\|Y(x,\tau) \ r\left(T(x,\tau)\right)\right\|_{L^\infty_{x}} \ \diff \tau
        \\
        &\leq\left(1+t\right)^{-\frac{d}{2}}  \eps_0 + C\left\|Y_0\right\|_{L^\infty_{x}} \eps_1^{11} \left[\int_0^{t/2} \left(1+t-\tau\right)^{-\frac{d}{2}} \left(1+\tau\right)^{-5d} \ \diff \tau
        +\int_{t/2}^t \left(1+\tau\right)^{-\frac{11}{2}d}\right]
    \\ &\leq \left(1+t\right)^{-\frac{d}{2}}\left[\eps_0 + C\left\|Y_0\right\|_{L^\infty_x} \eps_1^{11}\right].
        \end{align*}
        For $t\leq  1$ we may apply an almost identical approach, but we don't split the Duhamel term and we only use \eqref{eqn:weak_damping_estimate}. This takes care of the $L^\infty_x$ piece. The $L^1_{x}$ estimate is obtained by a similar but even easier argument using \eqref{eqn:weak_damping_estimate} to control $\left\|e^{\left(t-\tau\right)\Delta}Yr\left(T\right)\right\|_{L^1_x}$.
\end{proof}
\noindent Note that there is nothing particularly special about the power of $p=11$ appearing in the above proof: any large enough natural number will do, and the rapid vanishing of $r(T)$ at $T=0$ guarantees we can pick such a large $p$. 

\begin{thm}\label{thm:decay_for_lambda=0_nonradiating}
    There exists $\eps_0 =\eps_0\left(\left\|Y_0\right\|_{L^\infty_x}\right)\in (0,1)$ such that, under the hypothesis \eqref{eqn:smallness_hypothesis_4_burnout}, \eqref{eqn:model_ivp_full} admits a unique global-in-time mild solution $\left(T,Y\right)\in X_{\infty}\times C^0_t L^{\infty}_x$ obeying the decay estimate
    \begin{equation}
        \left\|T(x,t)\right\|_{L^{\infty}_{x}} \lesssim \eps_0 \left(1+t\right)^{-\frac{d}{2}}. 
    \end{equation}
    In particular, $\lim_{t\rightarrow \infty}|T(x,t)| = 0$ uniformly in $x$. 
\end{thm}
\begin{proof}
    Using fixed-point iteration, one finds a $t_{*}>0$ such that \eqref{eqn:model_ivp_full} has a unique local-in-time mild solution valid on $[0,t_*]$ with $T(x,t)\in X_{t_{*}}$. Further, the map $\tau\mapsto \left\|T\right\|_{\tau}$ is easily shown to be continuous for $\tau\in [0,t_*]$. Accordingly, for sufficiently small $t_{*}$ we know $\left\|T\right\|_{X_{t_{*}}}\leq 3\eps_0$. Now, apply proposition \ref{prop:bootstrap_close} with $\eps_1 = 3\eps_0$ to discover 
    \begin{align*}
        \left\|T\right\|_{X_{t_{*}}} \leq 2\eps_0 \left(1 + C \left\|Y_0\right\|_{L^\infty_x} \eps_0^{10}\right). 
    \end{align*}
    By choosing $\eps_0 < \frac12\left(C\left\|Y_0\right\|_{L^\infty_x}\right)^{-\frac{1}{10}}$, the above becomes $\left\|T\right\|_{X_{t_{*}}}\leq \frac52\eps_0$. The bootstrap principle (see for example \cite{Tao2006}) allows us to iterate this argument and produce a global solution satisfying $\left\|T\right\|_{X_{\infty}}\lesssim \eps_0$. 
\end{proof}
\noindent 
The $\eps_0$ from theorem \ref{thm:decay_for_lambda=0_nonradiating} should be regarded as a lower bound on the auto-ignition temperature for the model \eqref{eqn:model_ivp_full}, but my methods cannot determine how sharp this bound is. 

\section{Acknowledgements}
I'd like to thank Carrie Clark for helpful comments on the draft manuscript, and Dominic Shillingford for invaluable discussions. 

\bibliography{fire_modelling}
\bibliographystyle{plain}

\end{document}